\documentclass[12pt]{amsart}
\usepackage{amssymb}
\usepackage{amsmath}
\usepackage{amssymb}
\usepackage{amsthm}
\usepackage{enumerate}
\usepackage[all,arc,curve,color,frame]{xy}
\usepackage[utf8]{inputenc}

\DeclareMathOperator{\ad}{ad}

\newcommand{\A}{{\mathcal A}}

\newcommand{\E}{{\mathcal E}}
\newcommand{\F}{{\mathcal F}}
\renewcommand{\b}{{\mathfrak b}}
\renewcommand{\c}{{\mathfrak c}}
\newcommand{\g}{{\mathfrak g}}

\newcommand{\J}{{\mathcal J}}

\newcommand{\K}{{\mathcal K}}

\newcommand{\Ncal}{{\mathcal N}}
\newcommand{\C}{\ensuremath{\mathbb{C}}}
\newcommand{\R}{\ensuremath{\mathbb{R}}}

\newcommand{\p}{\partial}

\newcommand{\D}{{\mathcal D}}

\newcommand{\x}{{x_0}}

\newcommand{\Z}{\ensuremath{\mathbb{Z}}}

\newtheorem{lemma}{Lemma}[section]
\newtheorem{definition}{Definition}[section]
\newtheorem{proposition}{Proposition}[section]
\newtheorem{theorem}{Theorem}[section]

\begin{document}

\title[Formal oscillatory distributions]
{Formal oscillatory distributions}
\author[Alexander Karabegov]{Alexander Karabegov}
\address[Alexander Karabegov]{Department of Mathematics, Abilene
Christian University, ACU Box 28012, Abilene, TX 79699-8012}
\email{axk02d@acu.edu}


\maketitle
\begin{abstract}
We introduce the notion of an oscillatory formal distribution supported at a point. We prove that a formal distribution is given by a formal oscillatory integral if and only if it is an oscillatory distribution that has a certain nondegeneracy property. We give an algorithm that recovers the jet of infinite order of the integral kernel of a formal  oscillatory integral at the critical point from the corresponding formal distribution. We also prove that a star product $\star$ on a Poisson manifold $M$ is natural in the sense of Gutt and Rawnsley if and only if the formal distribution $f \otimes g \mapsto (f \star g)(x)$ is oscillatory for every $x \in M$.
\end{abstract}
\section{Introduction}

According to the stationary phase method, if $\phi$ is a real phase function on $\R^n$ which has a nondegenerate critical point $\x$ with zero critical value, $\phi(\x)=0$, and $f$ is an amplitude supported near $\x$, there exists an asymptotic expansion
\begin{equation}\label{E:stph}
   \left(\frac{i}{\hbar}\right)^{\frac{n}{2}} \int e^{\frac{i}{\hbar} \phi(x)} f(x) \, dx \sim \Lambda_0 (f) + \frac{\hbar}{i} \Lambda_1(f) + \left(\frac{\hbar}{i}\right)^2 \Lambda_2(f) +\ldots
\end{equation}
as $\hbar \to 0$, where $\Lambda_r$ are distributions supported at $\x$ (see \cite{L}). The formal distribution
\begin{equation}\label{E:foif}
   \Lambda = \Lambda_0 + \nu \Lambda_1 + \nu^2 \Lambda_2 + \ldots,
\end{equation}
where we use the formal parameter $\nu$ instead of $\hbar/i$, is a formal oscillatory integral (FOI) in the terminology of \cite{KS} and \cite{LMP10}. It can be defined by simple algebraic axioms expressed in terms of the jet of infinite order of the phase function $\phi$ at $\x$. Moreover, the full jet of $\phi$ at $\x$ is uniquely determined by the formal distribution $\Lambda$. We build an algorithm that allows to recover this jet of infinite order from $\Lambda$.

The class of FOIs introduced in \cite{KS} is more general. It includes the asymptotic expansions of oscillatory integrals where the phase function itself has an asymptotic expansion in $\hbar$ and can be complex, as explained in Section \ref{S:foi}.

In this paper we answer the following question asked by Th. Voronov: given a formal distribution, how to determine whether it is a FOI? To this end we introduce the notion of an oscillatory distribution. It is a formal distribution $\Lambda$ supported at a point $\x$ which in local coordinates is given by the formula
\[
    \Lambda(f) = e^{\nu^{-1}X} f \big|_{x=\x},
\] 
where $X= \nu^2 X_2 + \nu^3 X_3 + \ldots$ is a formal differential operator with constant coefficients such that the order of the differential operator $X_r$ is at most $r$ for all $r \geq 2$. It turns out that this property does not depend on the choice of local coordinates. We show that a formal distribution is a FOI if and only if it is an oscillatory distribution  that has a certain nondegeneracy property. 

In \cite{GR} Gutt and Rawnsley singled out an important class of star products which they call {\it natural}. For each $r \geq 1$, the bidifferential operator $C_r$ for a natural star product is of order at most $r$ in both arguments (see details in Section \ref{S:nat}). All classical star products are natural. We will prove that a star product $\star$ on a Poisson manifold $M$ is natural if and only if the formal distribution
\[
    \Lambda_x (f \otimes g) = (f \star g)(x)
\]
on $M^2$ supported at $(x,x)$ is oscillatory for every $x$. 

These results belong to the general framework of formal asymptotic Lagrangian analysis. Various semiclassical and quantum aspects of this analysis are developed in the work on formal symplectic groupoids by Cattaneo, Dherin, and Felder \cite{CDF} and the author \cite{CMP3}, symplectic microgeometry by Cattaneo, Dherin, and Weinstein \cite{CDW},   Lagrangian analysis by Leray \cite{L},  the theory of oscillatory modules by Tsygan \cite{T}, and microformal analysis by Th. Voronov \cite{V}.

{\bf Acknowledgements} I am very grateful to A. Alekseev, H. Khudaverdian, B. Tsygan, and Th. Voronov for important discussions and for the opportunity to present a part of this work at two conferences and during a visit to the University of Geneva in 2019.

\section{Factorization}

In this section we prove an elementary factorization result on pronilpotent Lie groups in filtered associative algebras which is the technical backbone of this paper. 

Let $\A$ be a filtered associative unital algebra over $\C$ with descending filtration $\A=\A_0 \supset\A_1 \supset \ldots$ such that $\bigcap_i \A_i = \{0\}$. We denote by $d(a)$ the filtration degree of $a \in \A$ so that $d(a)= k$ for $a \in \A_k \setminus A_{k+1}$. We assume that this algebra is complete with respect to the norm $|a| = 2^{-d(a)}$. Then any series $\sum_i a_i$ with $a_i \in \A$ such that $|a_i| \to 0$ is convergent. 

Let $\g \subset \A_1$ be a Lie algebra with respect to the commutator $[a,b]=ab-ba$. Then $\g$ is pronilpotent and $\exp \g \subset \A_0$ is the corresponding Lie group. 
Each element $g \in \exp \g$ is uniquely represented as
\begin{equation}\label{E:exp1}
g = \exp \gamma = \sum_{n=0}^\infty \frac{1}{n!}\gamma^n
\end{equation}
for some $\gamma \in \g$. Then $g-1 \in \A_1$ and
\begin{equation}\label{E:log1}
\gamma = \log (1- (1-g)) =-\sum_{n=1}^\infty \frac{1}{n}(1-g)^n.
\end{equation}

We set $\g_i := \g \cap \A_i$ for $i \geq 1$. The following statement is a consequence of formulas (\ref{E:exp1}) and (\ref{E:log1}).
\begin{lemma}\label{L:congr}
If $\gamma \in \g$, then $(\exp \gamma) - 1 \in \A_i$ if and only if $\gamma \in \g_i$.
\end{lemma}

Suppose that $\g$ is a direct sum of subalgebras $\mathfrak{a}$ and $\mathfrak{b}$ such that $\g_i=\mathfrak{a}_i \oplus \mathfrak{b}_i$, where $\mathfrak{a}_i := \mathfrak{a} \cap \A_i$ and $\mathfrak{b}_i := \mathfrak{b} \cap \A_i$, for all $i \geq 1$.
\begin{proposition}\label{P:factor}
 Any element $g \in \exp \g$ can be uniquely factorized as $g=ab$ with $a \in \exp \mathfrak{a}$ and $b \in \exp \mathfrak{b}$.
\end{proposition}
\begin{proof}
Given  $g = \exp \gamma_0 \in \exp \g$ for some $\gamma_0 \in \g_1=\g$, we can represent $\gamma_0$ uniquely as $\gamma_0= \alpha_0+\beta_0$ for some $\alpha_0 \in  \mathfrak{a}_1$ and $\beta_0 \in  \mathfrak{b}_1$.  It follows from Lemma \ref{L:congr} that
\[
   e^{-\alpha_0} e^{\gamma_0} e^{-\beta_0} = e^{\gamma_1}
\]
for some $\gamma_1 \in \g_2$. Then $\gamma_1 = \alpha_1+\beta_1$ for $\alpha_1\in \mathfrak{a}_2$ and $\beta_1 \in \mathfrak{b}_2$. Repeating this process, we obtain sequences $\{\alpha_i\}, \{\beta_i\}$, and $\{\gamma_i\}$ with $\alpha_i \in \mathfrak{a}_{2^i}, \beta_i \in \mathfrak{b}_{2^i}$, and $\gamma_i \in \mathfrak{g}_{2^i}$ such that $\gamma_i=\alpha_i+\beta_i$ and
\[
  e^{-\alpha_i} e^{\gamma_i} e^{-\beta_i} = e^{\gamma_{i+1}}.
\]
We get that
\[
g = e^{\gamma_0} = e^{\alpha_0} e^{\gamma_1}e^{\beta_0} = e^{\alpha_0}e^{\alpha_1} e^{\gamma_2} e^{\beta_1}e^{\beta_0} =\ldots
\]
It follows that $g=ab$, where $a \in \exp \mathfrak{a}$ and $b \in \exp \mathfrak{b}$ are given by the convergent infinite products
\[
    a = e^{\alpha_0}e^{\alpha_1} e^{\alpha_2} \ldots \mbox{ and } b = \ldots e^{\beta_2} e^{\beta_1}e^{\beta_0}.
\]
The representation $g=ab$ is unique because $\exp \mathfrak{a} \cap \exp \mathfrak{b}=\{1\}$.
\end{proof}

Throughout this paper, we will apply Proposition \ref{P:factor} several times in different contexts. Each time we will reuse the same notations for a filtered associative algebra $\A$ and a pronilpotent Lie algebra $\g \subset \A_1$.

\section{Some classes of formal distributions and operators}

Let $M$ be a real manifold and $\x$ be a point in $M$. We denote by $\mathbb{D}(M)$ the algebra of differential operators on $M$, by $\mathbb{D}_{\x}(M)$ the space of all distributions on $M$ supported at $\x$, and by $\delta_{\x}$ the Dirac distribution at $\x$ ($\delta_{\x} (f)=f(\x)$). The mapping
\[
     A \mapsto \delta_{\x} \circ A, 
\]
from $\mathbb{D}(M)$ to $\mathbb{D}_{\x}(M)$ is surjective.

Let $\nu$ be a formal parameter. We say that a $\nu$-formal differential operator
\[
  A = A_0 + \nu A_1 + \ldots \in \mathbb{D}(M)[[\nu]]
\]
 is {\it natural} if the order of $A_r$ is at most $r$ for all $r\geq 0$. If $U$ is a coordinate chart on $M$ with coordinates $\{x^i\}$, a natural operator $A$ on $U$ can be uniquely written as
\[
   A = \sum_{r=0}^\infty f_r^{i_1 \ldots i_r}(\nu,x)\left(\nu \p_{i_1}\right) \ldots \left(\nu \p_{i_r}\right),
\]
where $f_r^{i_1 \ldots i_r} \in C^\infty(U)[[\nu]]$ is symmetric in $i_1, \ldots, i_r$ for each $r \geq 0$ and $\p_i = \p/\p x^i$.

The natural operators on $M$ form an associative algebra. If $A$ and $B$ are natural operators, then the operator $\nu^{-1}[A,B]$ is natural. Therefore, the formal differential operators of the form $\nu^{-1}A$, where $A$ is natural, form a Lie algebra with respect to the commutator $[A,B]=AB-BA$.

\begin{definition}
A formal differential operator $A \in \mathbb{D}(M)[[\nu]]$ is called oscillatory if it is represented as $A=\exp(\nu^{-1} X)$, where $X= \nu^2 X_2 + \nu^3 X_3 + \ldots$ is a natural operator.
\end{definition}
\begin{definition}
A formal distribution $\Lambda  \in \mathbb{D}_{\x}(M)[[\nu]]$ is called oscillatory if there exists an oscillatory operator $A$ such that $\Lambda=\delta_{\x}\circ A$.
\end{definition}

Assume that $\Lambda = \Lambda_0 + \nu \Lambda_1 + \ldots$ is an oscillatory distribution  on~$M$ supported at $\x$ and represented as $\Lambda=\delta_{\x}\circ \exp (\nu^{-1}X)$, where $X= \nu^2 X_2 + \nu^3 X_3 + \ldots$ is natural. Then $\Lambda_0=\delta_{\x}$ and $\Lambda_1 = \delta_{\x} \circ X_2$. Since $X_2$ is a differential operator of order at most~2, there exists a unique symmetric bilinear form $\beta_{\Lambda}$ on $T_{\x}^\ast M$ such that
\[
     \beta_{\Lambda} (df(\x),dg(\x)) = \Lambda_1(fg)
\]
for any functions $f$ and $g$ on $M$ such that $f(\x)=g(\x)=0$.
The form $\beta_\Lambda$ is a coordinate-free object. Let $U \subset M$ be a coordinate neighborhood of $\x$  with coordinates $\{x^i\}$. If $X_2 = a^{ij} \p_i\p_j + b^i\p_i + c$, then
\[
   \beta_{\Lambda} (df(\x),dg(\x)) = 2 a^{ij} \p_i f \p_j g \big|_{x=\x}.
\]
The form $\beta_{\Lambda}$ is thus given by the tensor $2a^{ij}(\x)$.
\begin{definition}
An oscillatory distribution $\Lambda$ is called nondegenerate if the bilinear form $\beta_{\Lambda}$ is nondegenerate.
\end{definition}

If $\Lambda$ is a distribution on a coordinate neighborhood $U$ of $\x$ supported at $\x$, there exists a unique differential operator $C$ with constant coefficients such that $\Lambda=\delta_{\x}\circ C$. 
 We will need the following fact. 

\begin{lemma}\label{L:repr}
Any differential operator $A$ on $U$ can be uniquely represented as a sum $A=B+C$ of differential operators such that $\delta_{\x} \circ B=0$ and $C$ has constant coefficients.
\end{lemma}
\begin{proof}
Let $C$ be the unique differential operator with constant coefficients such that
\[
     \delta_{\x} \circ C = \delta_{\x} \circ A.
\]
Set $B:= A-C$. Then $\delta_{\x} \circ B=0$ and $A=B+C$.
\end{proof}

Any differential operator $A$ on $U$ can be uniquely represented in the normal form,
\[
     A = \sum_{r=0}^N  A^{i_1 \ldots i_r}(x)  \p_{i_1} \ldots \p_{i_r},
\]
where $A^{i_1 \ldots i_r}(x) \in C^\infty(U)$ is symmetric in $i_1, \ldots,i_r$. Then $A=B+C$, where
\[
    B =  \sum_{r=0}^N  \left(A^{i_1 \ldots i_r}(x) - A^{i_1 \ldots i_r}(\x)\right)   \p_{i_1} \ldots \p_{i_r}
\]
is such that $\delta_{\x} \circ B=0$ and
\[
     C=\sum_{r=0}^N  A^{i_1 \ldots i_r}(\x)  \p_{i_1} \ldots \p_{i_r}
\]
has constant coefficients. 

We fix a coordinate chart $U$ and consider the algebra $\A : = \mathbb{D}(U)[[\nu]]$ of formal differential operators on $U$ equipped with the $\nu$-filtration (the filtration degree of $\nu$ is~1).
 Let $\g \subset \A_1$ be the Lie algebra of formal differential operators on $U$ of the form $\nu^{-1} X$, where $X= \nu^2 X_2 + \nu^3 X_3 + \ldots$ is a natural operator. This is a pronilpotent Lie algebra with respect to the $\nu$-filtration. A distribution $\Lambda$ on $U$ supported at a point $\x$ is oscillatory if there exists an element $A \in \g$ such that $\Lambda = \delta_{\x} \circ \exp (A)$. The following proposition provides a criterion that a given formal distribution supported at a point is oscillatory.

\begin{proposition}\label{P:osciff}
 Let $\Lambda$ be a formal distribution on $U$ supported at a point $\x$. If $C$ is the unique formal differential operator with constant coefficients such that
 \[
\Lambda = \delta_{\x} \circ \exp (C),
\]
 then $\Lambda$ is oscillatory if and only if $C \in \g$.
\end{proposition}
\begin{proof}
If $C \in \g$, then $\Lambda$ is oscillatory. Now assume that $\Lambda$ is oscillatory. Let $\b$ be the Lie algebra of formal differential operators $A \in \g$ such that $\delta_{\x} \circ A=0$. Denote by $\c$ the Lie algebra of the formal differential operators with constant coefficients from $\g$. Lemma \ref{L:repr} implies that $\g=\b\oplus \c$ and $\g_i=\b_i\oplus \c_i$ for all $i \geq 1$ for the corresponding $\nu$-filtration spaces. Notice that the algebras $\g$ and $\b$ are coordinate-free objects, while the complementary algebra $\c$ depends on the choice of coordinates on $U$. Since $\Lambda$ is oscillatory, $\Lambda = \delta_{\x} \circ \exp (A)$ for some $A \in \g$. It follows from Proposition \ref{P:factor} that there exist unique elements $B \in \b$ and $C \in \c$ such that $e^A = e^B e^C$. Then $\delta_{\x} \circ \exp B=\delta_{\x}$ and
\[
    \Lambda = \delta_{\x} \circ \exp (A) = \delta_{\x} \circ \left(\exp (B) \exp(C)\right) = \delta_{\x} \circ \exp (C).
\]
\end{proof}

\section{Natural star products}\label{S:nat}

 Given a vector space $V$, we denote by $V((\nu))$ the space of formal vectors
\[
    v = \nu^r v_r + \nu^{r+1} v_{r+1} + \ldots,
\]
where $r \in \Z$ and $v_i \in V$ for all $i \geq r$.

Let $M$ be a Poisson manifold with Poisson bracket $\{\cdot,\cdot\}$. A star product $\star$ on $M$ is an associative product on $C^\infty(M)((\nu))$ given by the formula
\begin{equation}\label{E:star}
   f \star g = fg + \sum_{r=1}^\infty \nu^r C_r(f,g),
\end{equation}
where $C_r$ are bidifferential operators on $M$ for $r \geq 1$ and $C_1(f,g)-C_1(g,f)=\{f,g\}$ (see \cite{BFFLS}). We assume that the unit constant 1 is the unity for the star product, $f \star 1 = f = 1 \star f$ for all $f$. Given $f,g \in C^\infty(M)((\nu))$, denote by $L_f$ the operator of left star multiplication by $f$ and by $R_g$ the operator of right star multiplication by $g$ so that
\[
    L_f g = f \star g = R_g f.
\]
The associativity of the star product $\star$ is equivalent to the condition that $[L_f,R_g]=0$ for any $f,g$. The mapping $f \mapsto L_f$ is an injective homomorphism from the star algebra $(C^\infty(M)((\nu)), \star)$ to the algebra $\mathbb{D}(M)((\nu))$ of formal differential operators on $M$. It has a left inverse mapping $A \mapsto A1$ (which is not a homomorphism on the whole algebra $\mathbb{D}(M)((\nu))$),
\[
L_f \mapsto L_f 1 = f \star 1 = f.
\]

Gutt and Rawnsley introduced in \cite{GR} an important notion of a natural star product. A star product (\ref{E:star}) is natural if the bidifferential operator $C_r$ is of order not greater than $r$ in both arguments for every $r \geq 1$. Equivalently, a star product $\star$ is natural if the operators $L_f$ and $R_f$ are natural for all $f \in C^\infty(M)$. Then $L_f$ and $R_f$ are natural for all $f \in C^\infty(M)[[\nu]]$. All classical star products (Moyal-Weyl, Wick, Fedosov, and Kontsevich star products) are natural (see \cite{GR}, \cite{F1}, and \cite{K}). We give an equivalent description of natural star products in terms of oscillatory distributions in Theorem \ref{T:grosc} below. To prove this theorem, we need some preparations.

Let $t_1, \ldots, t_n$ be formal parameters, where $n$ is any number,  and
\[
\A:= \left(\mathbb{D}(M)((\nu))\right)[[t_1, \ldots, t_n]]
\]
be the associative algebra of formal differential operators on $M$ of the form
\begin{equation}\label{E:nutop}
   A=   \sum_{k=0}^\infty t_{j_1} \ldots t_{j_k} A^{j_1\ldots j_k},
\end{equation}
where $A^{j_1\ldots j_k} \in \mathbb{D}((\nu))$ are $\nu$-formal differential operators on $M$ symmetric in $j_1, \ldots,j_k$. We equip $\A$ with the $t$-filtration $\{\A_i\}$ for which the filtration degree of $t_i$ is 1 for every $i$ (and the filtration degree of $\nu$ is zero).  We say that an operator (\ref{E:nutop}) is natural if all operators $A^{j_1\ldots j_k}$ are natural. The algebra $\A$ acts on the space $\F:= \left(C^\infty(M)((\nu))\right)[[t_1, \ldots,t_n]]$ equipped with the $t$-filtration $\{\F_i\}$. The space $\F$ is a commutative algebra with respect to the ``pointwise" multiplication of formal series. Given $f\in \F$, we denote by $m_f$ the multiplication operator by $f$. Then $m_f \in \A$ and $m_f 1=f$. Each operator $A\in \A$ is uniquely represented as the sum
\begin{equation}\label{E:abc}
   A = m_{A1} + (A-m_{A1}),
\end{equation}
where $A-m_{A1}$ annihilates constants, $(A-m_{A1})1=0$.

Let $\g \subset \A_1$ be the Lie algebra of operators of positive $t$-filtration degree of the form $\nu^{-1}A$, where $A \in \A$ is natural. The Lie algebra $\g$ is pronilpotent with respect to the $t$-filtration $\{\g_i\}$, where $i \geq 1$. Its Lie group is $\exp\g \subset \A_0$.

Denote by $\mathfrak{a}$ the commutative subalgebra of $\g$ of multiplication operators and by $\mathfrak{b}$ the subalgebra of $\g$ of operators that annihilate constants. Then 
$\g = \mathfrak{a}\oplus\mathfrak{b}$ and $\g_i = \mathfrak{a}_i\oplus\mathfrak{b}_i$ for all $i \geq 1$ in accordance with the representation (\ref{E:abc}). Let $\mathcal{G}$ be the set of formal functions
\[
   f=\nu^{-1}f_{-1} + f_0 +\nu f_1 + \ldots
\] 
from $\F_1$. Then $\mathfrak{a}=\{m_f | f \in \mathcal{G}\}$. Given $f \in \mathcal{G}$, the exponential series
\[
   e^f = 1 + f + \frac{1}{2} f^2 + \ldots 
\]
defines an element of $\F_0$ and $\exp\mathfrak{a}= \{m_{e^f}|f \in \mathcal{G}\}$. We set
\[
\exp\mathcal{G} := \{ e^f| f \in  \mathcal{G}\} \subset \F_0.
\]
It is the Lie group of the commutative Lie algebra $\mathcal{G}$. The mapping $a \mapsto a1$ is a group isomorphism from $\exp \mathfrak{a}$ onto $\exp\mathcal{G}$.
\begin{lemma}\label{L:expfone}
For each $g \in \exp \g$, the operator $g$ leaves invariant the set $\exp  \mathcal{G}$. In particular, $g1 \in \exp  \mathcal{G}$.
\end{lemma}

\begin{proof}
Assume that $g \in \exp \g$ and $f \in \mathcal{G}$. Then $m_{e^f} \in\exp \mathfrak{a}$ and $gm_{e^f}\in\exp\g$. By Proposition \ref{P:factor}, the element $gm_{e^f}$ is uniquely represented as a product $gm_{e^f}= ab$, where $a \in \exp \mathfrak{a}$ and $b \in \exp \mathfrak{b}$. Then $a1 \in \exp\mathcal{G}$ and $b1=1$. Therefore, applying the operator $g$ to the function $e^f$, we get
\[
g (e^f) = (gm_{e^f})1 = (ab)1  =a1 \in  \exp\mathcal{G}.
\]
Thus, $g(\exp\mathcal{G}) \subset \exp\mathcal{G}$ and therefore $g1 \in \exp\mathcal{G}$.
\end{proof}

Let $\star$ be a natural star product on $M$. We extend it to $\F$ so that $L_{t_i}=R_{t_i}=t_i$ be the ``pointwise" multiplication operator by $t^i$ for every~$i$. The space $\mathcal{G}\subset \F_1$
 is a Lie algebra with respect to the star-commutator $[f,g]_\star = f \star g - g \star f$. This Lie algebra is pronilpotent with respect to the $t$-filtration $\{\mathcal{G}_i\}$, where $i \geq 1$. Given $f \in \mathcal{G}$, the exponential series
\[
    \exp_{\star} f= 1 + f + \frac{1}{2} f \star f + \ldots
\]
defines an element of $\F_0$. We set
\[
    \exp_{\star} \mathcal{G} := \{ \exp_{\star} f | f \in \mathcal{G}\} \subset \F_0.
\]
This is the Lie group of the Lie algebra $(\mathcal{G}, [\cdot, \cdot]_{\star})$.
\begin{lemma}\label{L:etof}
The subsets $\exp_{\star} \mathcal{G}$ and $\exp \mathcal{G}$ of $\F_0$ coincide.
\end{lemma}
\begin{proof}
Given $f \in \mathcal{G}$, the operator $\nu L_f = L_{\nu f}$ is natural and therefore $L_f \in \g$. Thus, $\exp L_f \in \exp\g$. By Lemma \ref{L:expfone}, the operator $\exp L_f $ with $f \in \mathcal{G}$ leaves invariant the set $\exp\mathcal{G}$. Given $f, g \in \mathcal{G}$, we have
\begin{equation*}
   (\exp_{\star} f) \star e^g = \left(L_{\exp_{\star} f}\right) e^g = \left(\exp {L_f} \right)e^g\in \exp\mathcal{G}.
\end{equation*}
Taking $g=0$, we get that $\exp_{\star} f \in \exp\mathcal{G}$. Hence, $\exp_{\star} \mathcal{G} \subset \exp\mathcal{G}$. 

 Given $u \in \mathcal{G}_i$, there exists $v \in \mathcal{G}$ such that $e^v = \exp_{\star} (-u) \star e^u$. Since
\[
     e^u = 1 + u \pmod {\F_{2i}} \mbox{ and } \exp_{\star} (-u) = 1 - u \pmod{\F_{2i}},
\]
we see that $e^v \in 1 + \F_{2i}$ and therefore $v \in \mathcal{G}_{2i}$. 

Let $f \in \mathcal{G}=\mathcal{G}_1$. We will show that $e^f \in \exp_{\star} \mathcal{G}$.
We construct a sequence $\{f_k\}, k \geq 0$, in $\mathcal{G}$ such that $f_0=f \in \mathcal{G}_1$ and
\[
    e^{f_{k+1}} = \exp_{\star} (-f_k) \star e^{f_k}
\]
for $k \geq 0$. We have $f_k \in \mathcal{G}_{2^k}$ for all $k \geq 0$. Observe that
\[
   e^f = e^{f_1} = (\exp_\star f_1) \star e^{f_2} = (\exp_\star f_1) \star (\exp_\star f_2) \star e^{f_3} = \ldots
\]
Since $e^{f_k} \to 1$ as $k \to \infty$ in the topology induced by the $t$-filtration, we get that
\[
   e^f =  (\exp_\star f_1) \star (\exp_\star f_2)  \star \ldots \in \exp_\star \mathcal{G}.
\]
It follows that $\exp_\star \mathcal{G} = \exp \mathcal{G}$.
\end{proof}

We give some basic facts on full symbols of formal differential operators. Let $U$ be a coordinate chart with coordinates $\{x^i\}, i=1,\ldots,n$, and let $\{\xi_i\}$ be the dual fiber coordinates on $T^\ast U$ which are treated as formal parameters. A formal differential operator $A \in \mathbb{D}(U)((\nu))$ can be written in the normal form as
\[
     A = \sum_{j=k}^\infty \nu^j \sum_{r=0}^{N_j}   A_j^{i_1 \ldots i_r}(x) \p_{i_1} \ldots \p_{i_r},
\]
where $k \in \Z, A_j^{i_1 \ldots i_r}(x) \in C^\infty(U)$ is symmetric in $i_1, \ldots, i_r$ for all~$j$ and~$r$, and $\p_i =\p/\p x^i$. The full symbol of the operator $A$ is the formal series
\[
     S(A) = \sum_{j=k}^\infty \sum_{r=0}^{N_j}  \nu^{j-r}  A_j^{i_1 \ldots i_r}(x) \xi_{i_1} \ldots \xi_{i_r},
\]
which is an element of $\left(C^\infty(U)((\nu))\right) [[\xi_1, \ldots,\xi_n]]$, because for a fixed $r$ the power of $\nu$ is bounded below by $k-r$.
The operator $A$ is natural if and only if $N_j\leq j$ for all $j$ or, equivalently, 
$S(A)$ does not contain negative powers of $\nu$. It is well-known that
\begin{equation}\label{E:symbol}
    S(A) = e^{-\frac{1}{\nu} x^i\xi_i} A \left(e^{\frac{1}{\nu} x^i\xi_i}\right) = \left(e^{-\frac{1}{\nu} x^i\xi_i} A e^{\frac{1}{\nu} x^i\xi_i}\right)1.
\end{equation}
The $\C((\nu))$-linear mapping $A \mapsto S(A)$ restricted to the formal differential operators with constant coefficients is an algebra homomorphism: if $A$ and $B$ have constant coefficients, then  $S(AB)=S(A)S(B)$.

For every $x\in M$ there exists a formal distribution $\Lambda_x$ on $M^2$ supported at $(x,x)$ such that
\begin{equation*}\label{E:lamx}
   \Lambda_x(f \otimes g) = (f \star g)(x)
\end{equation*}
for all $f,g \in C^\infty(M)$.

\begin{theorem}\label{T:grosc}
A star product $\star$ on a manifold $M$ is natural if and only if the formal distribution $\Lambda_x$ is oscillatory for all $x \in M$.
\end{theorem}

{\it Example}. Let $(\pi^{ij})$ be an $n \times n$ matrix with constant coefficients. The star product
\[
  f \star g = \sum_{r=0}^\infty \frac{\nu^r}{r!} \pi^{i_1 j_1} \ldots \pi^{i_r j_r}\frac{\p^r f}{\p x^{i_1} \ldots \p x^{i_r}} \frac{\p^r g}{\p x^{j_1} \ldots \p x^{j_r}}
\]
on $\R^n$ is natural. If the matrix $(\pi^{ij})$ is skew-symmetric and nondegenerate, this is the Moyal-Weyl star product. Consider the natural operator
\[
   A := \nu^2 \pi^{ij} \frac{\p^2}{\p y^i \p z^j}
\]
on $\R^{2n}$. The formula
\[
   \Lambda_x(f \otimes g) = (f \star g)(x) = e^{\nu^{-1} A} (f(y)g(z)) \big |_{y=z=x},
\]
where $f,g \in C^\infty(\R^n)$, shows that the formal distribution $\Lambda_x$ is oscillatory for any $x$. It is nondegenerate if and only if the matrix $(\pi^{ij})$ is nondegenerate.

Now we proceed with a proof of Theorem \ref{T:grosc}.

\begin{proof}
Assume that a star product $\star$ on $M$ is such that the distribution $\Lambda_x$ is oscillatory for all $x \in M$. Let $U$ be a coordinate chart on $M$ with coordinates $\{x^i\}$. Then for each $x \in U$  there exists a unique natural operator with constant coefficients
\begin{equation}\label{E:aofx}
   A(x) = \sum_{r=2}^\infty \nu^r \sum_{k+l \leq r} F_{r,k,l}^{i_1 \ldots i_k j_1 \ldots j_l}(x)\frac{\p^k}{\p y^{i_1} \ldots \p y^{i_k}}\frac{\p^l}{\p z^{j_1} \ldots \p z^{j_l}}
\end{equation}
such that
\begin{equation}\label{E:expstar}
   (f \star g)(x) = e^{\nu^{-1} A(x)} (f(y)g(z))\big|_{y=z=x}.
\end{equation}
Since $(f \star 1)(x)=f(x)$, we get that
\[
    \exp \left(\sum_{r=2}^\infty \nu^{r-1} \sum_{k \leq r} F_{r,k,0}^{i_1 \ldots i_k}(x)\frac{\p^k}{\p y^{i_1} \ldots \p y^{i_k}}\right) f(y) \bigg |_{y=x} = f(x)
\]
for any $f(x)$. Hence, $F_{r,k,0}^{i_1 \ldots i_k}(x)=0$ for all $r$ and $k$. Similarly, $F_{r,0,l}^{j_1 \ldots j_l}(x)=0$ for all $r$ and $l$. 

Given $f\in C^\infty(U)$, we will prove that the operator $L_f$ is natural. To this end, we will calculate its full symbol $S(L_f)$ using (\ref{E:symbol}) and (\ref{E:expstar}). We will show that it does not contain negative powers of $\nu$. We have
\begin{eqnarray*}
S(L_f) = e^{-\nu^{-1} x^i \xi_i} L_f \left(e^{\nu^{-1} x^i \xi_i}\right) = e^{-\nu^{-1} x^i \xi_i}  \left(f \star e^{\nu^{-1} x^i \xi_i} \right)=\\
e^{-\nu^{-1} x^i \xi_i} e^{\nu^{-1} A(x)} \left(f(y) e^{\nu^{-1} z^i \xi_i}\right) \bigg |_{y=z=x} =\hskip 2.5cm \\
\left(e^{-\nu^{-1} z^i \xi_i} e^{\nu^{-1} A(x)} e^{\nu^{-1} z^i \xi_i}\right) f(y) \bigg |_{y=z=x} =\hskip 1.5cm\\
\exp \left(e^{-\nu^{-1} z^i \xi_i}\left(\nu^{-1} A(x)\right)e^{\nu^{-1} z^i \xi_i}\right) f(y) \bigg |_{y=z=x}.
\end{eqnarray*}
It suffices to prove that the operator $e^{-\nu^{-1} z^i \xi_i}\left(\nu^{-1} A(x)\right)e^{\nu^{-1} z^i \xi_i}$ does not contain negative powers of $\nu$. Using (\ref{E:aofx}), we will write this operator as follows,
\[
 \sum_{r=2}^\infty \nu^{r-1} \sum_{k+l \leq r} F_{r,k,l}^{i_1 \ldots i_k j_1 \ldots j_l}\frac{\p^k}{\p y^{i_1} \ldots \p y^{i_k}}\left(\frac{\p}{\p z^{j_1}}+ \frac{1}{\nu}\xi_{i_1}\right) \ldots \left(\frac{\p}{\p z^{j_l}}+ \frac{1}{\nu}\xi_{i_l}\right).
\]
Since $F_{r,0,l}^{j_1 \ldots j_l}=0$ for all $r$ and $l$, the condition $k+l \leq r$ in the second sum implies that $l \leq r-1$, which proves the claim. One can show similarly that the operator $R_f$ is natural for $f\in C^\infty(U)$. Since $U$ is arbitrary, the star product $\star$ is natural on $M$.

Now assume that $\star$ is a natural star product on $M$ and $U \subset M$ is an arbitrary coordinate chart. We will show that $\Lambda_x$ is oscillatory for every $x\in U$.
Let $\{\xi_i\}$ and $\{\eta_i\}$ be two sets of formal variables dual to $\{x^i\}$. We extend the star product $\star$ to $\F:=(C^\infty(U)((\nu)))[[\xi,\eta]]$ so that $L_{\xi_i}= R_{\xi_i} = \xi_i$ and $L_{\eta_i}= R_{\eta_i} = \eta_i$ for all~$i$. Denote by $\mathcal{G}$ the Lie algebra of functions from $\nu^{-1}C^\infty(U)[[\nu, \xi,\eta]]$  of positive filtration degree with respect to the variables $\xi$ and $\eta$ with the star commutator $[f,g]_\star = f \star g - g \star f$ as  the Lie bracket. This is a pronilpotent Lie algebra with the Lie group $\exp_{\star} \mathcal{G}$ whose elements are the star exponentials
\[
    \exp_{\star} f = 1 + f + \frac{1}{2} f \star f + \ldots 
\]
of the elements of $\mathcal{G}$. We can write the star product $\star$ as (\ref{E:expstar}) with
\[
    A(x) = \sum_{r=2}^\infty \nu^r \sum_{k+l \leq N_r} F_{r,k,l}^{i_1 \ldots i_k j_1 \ldots j_l}(x)\frac{\p^k}{\p y^{i_1} \ldots \p y^{i_k}}\frac{\p^l}{\p z^{j_1} \ldots \p z^{j_l}},
\]
where $N_r$ is some integer for each $r \geq 2$. We have to show that $A(x)$ is natural for every $x\in U$, i.e., that $N_r \leq r$ for all $r\geq 2$. To this end, we consider two functions in $\exp \mathcal{G}=\{ e^f | f \in \mathcal{G}\}$,
\[
   f(x):= e^{\nu^{-1} x^i \xi_i} \mbox{ and }  g(x):=e^{\nu^{-1} x^i \eta_i} .
\]
By Lemma \ref{L:etof}, $f,g \in \exp_{\star} \mathcal{G}$. Therefore, $f \star g \in \exp_{\star} \mathcal{G}=\exp \mathcal{G}$.
Using (\ref{E:symbol}) and  (\ref{E:expstar}), we get that for $x\in U$,
\begin{eqnarray*}
  (f \star g)(x)= e^{\nu^{-1} A(x)} \left(e^{\nu^{-1}(y^i \xi_i + z^i\eta_i)}\right)\big|_{y=z=x}=\hskip 3cm\\
  e^{\nu^{-1}x^i (\xi_i+\eta_i) } \left(e^{-\nu^{-1}(y^i \xi_i + z^i\eta_i)}e^{\nu^{-1} A(x)} e^{\nu^{-1}(y^i \xi_i + z^i\eta_i)}\right)1\big|_{y=z=x}=\\
  e^{\nu^{-1}x^i (\xi_i+\eta_i) } S\left(e^{\nu^{-1} A(x)}\right) = e^{\nu^{-1} \left(x^i (\xi_i+\eta_i) + S(A(x))\right)} \in \exp\mathcal{G},
\end{eqnarray*}
where
\[
  S(A(x)) =\sum_{r=2}^\infty  \sum_{k+l \leq N_r} \nu^{r-k-l} F_{r,k,l}^{i_1 \ldots i_k j_1 \ldots j_l}(x)\xi_{i_1} \ldots \xi_{i_k} \eta_{j_1} \ldots \eta_{j_l}
\]
is the full symbol of $A(x)$. Since
\[
\nu^{-1} \left(x^i (\xi_i+\eta_i) + S(A(x))\right) \in \mathcal{G}, 
\]
$S(A(x))$ does not contain negative powers of $\nu$, which implies that $A(x)$ is natural and therefore $\Lambda_x$ is oscillatory for any $x \in U$. Since $U$ is arbitrary, $\Lambda_x$ is oscillatory for any $x \in M$.
\end{proof}

In \cite{CMP3} it was shown that the natural star products have a good semiclassical behavior. Theorem \ref{T:grosc} relates these star products to oscillatory distributions which can be thought of as quantum objects.

\section{Formal oscillatory integrals}\label{S:foi}

Let $M$ be a real $n$-dimensional manifold, $\x$ be a point in $M$,
\[
    \varphi = \nu^{-1} \varphi_{-1} + \varphi_0 + \nu \varphi_1 + \ldots
\]
be a formal complex-valued function and  $\rho = \rho_0 + \nu \rho_1 + \ldots$ be a formal complex-valued density on $M$ such that $\x$ is a nondegenerate critical point of $\varphi_{-1}$ with zero critical value, $\varphi_{-1}(\x)=0$, and $\rho_0(\x)\neq 0$. We call the pair $(\varphi,\rho)$ a phase-density pair with the critical point $\x$.
A formal oscillatory integral (FOI) at $\x$ associated with the phase-density pair $(\varphi,\rho)$ is a formal distribution
\[
   \Lambda=\Lambda_0 + \nu \Lambda_1 + \ldots
\]
on $M$ supported at $\x$ such that the value $\Lambda(f)$ for an amplitude $f$ heuristically corresponds to the formal integral expression
\begin{equation}\label{E:fint}
   \nu^{-\frac{n}{2}}  \int e^\varphi f \rho.
\end{equation}
The distribution $\Lambda$ is defined by certain algebraic axioms expressed in terms of the pair $(\varphi,\rho)$ which correspond to formal integral properties of (\ref{E:fint}). The full stationary phase expansion of an oscillatory integral (\ref{E:stph}) whose amplitude is supported near a nondegenerate critical point of the phase function is given by a FOI. The notion of a FOI was introduced in \cite{KS} and developed further in \cite{LMP10}.

\begin{definition}
Given a phase-density pair $(\varphi,\rho)$ with a critical point~$\x$ on a manifold $M$, a formal distribution $\Lambda=\Lambda_0 + \nu \Lambda_1 + \ldots$ on $M$ supported at $\x$ and such that $\Lambda_0$ is nonzero is called a formal oscillatory integral (FOI) associated with the pair $(\varphi,\rho)$ if
\begin{equation}\label{E:axiom}
    \Lambda(vf + (v \varphi + \mathrm{div}_\rho v)f)=0
\end{equation}
for any function $f$ and any vector field $v$ on $M$.
\end{definition}
In (\ref{E:axiom})  $\mathrm{div}_\rho v$ denotes the divergence of the vector field $v$ with respect to $\rho$ given by the formula
\[
     \mathrm{div}_\rho v = \frac{\mathbb{L}_v \rho}{\rho},
\]
where $\mathbb{L}_v$ is the Lie derivative with respect to $v$. Axiom (\ref{E:axiom}) corresponds to the formal integral property
\[
    \nu^{-\frac{n}{2}}  \int \mathbb{L}_v(e^\varphi f \rho)=0.
\]
Observe that the condition (\ref{E:axiom}) is coordinate-independent. As shown in~\cite{LMP10}, a FOI $\Lambda$ associated with $(\varphi,\rho)$ satisfies the following properties.
\begin{enumerate}
\item $\Lambda$ exists and is unique up to a multiplicative formal constant $c = c_0 + \nu c_1 + \ldots$ with $c_0 \neq 0$.
\item $\Lambda_0 = \alpha \delta_{\x}$ for some nonzero complex constant $\alpha$.
\item  $\Lambda$ is determined by the jets of infinite order of $\varphi$ and $\rho$ at~$\x$.
\item If $u=u_0 + \nu u_1 + \ldots$ is any formal function on $M$, then $\Lambda$ is associated with $(\varphi+u, e^{-u}\rho)$.
\item If $\Lambda$ is associated with two pairs $(\varphi,\rho)$ and $(\tilde\varphi,\rho)$ which share the density $\rho$, then the full jet of $\tilde\varphi - \varphi$ at $\x$ is a formal constant.
\end{enumerate}

\begin{definition}
A FOI associated with a pair $(\varphi,\rho)$ is strongly associated with it if
\begin{equation}\label{E:strong}
      \frac{d}{d\nu}\Lambda(f) - \Lambda \left(\frac{df}{d\nu} + \left(\frac{d\varphi}{d\nu} + \frac{d\rho/d\nu}{\rho}- \frac{n}{2\nu}\right)f \right)=0
 \end{equation}
 for any function $f$.
\end{definition}
The condition (\ref{E:strong}) is coordinate-independent. It corresponds to the formal property of (\ref{E:fint}) that integration commutes with
differentiation with respect to the formal parameter $\nu$. A FOI $\Lambda$ strongly associated with $(\varphi,\rho)$ satisfies the following properties.
\begin{enumerate}
\item $\Lambda$ exists and is unique up to a multiplicative nonzero complex constant.
\item  $\Lambda$ is determined by the jets of infinite order of $\varphi$ and $\rho$ at $\x$.
\item If $u=u_0 + \nu u_1 + \ldots$ is any formal function on $M$, then $\Lambda$ is strongly associated with $(\varphi+u, e^{-u}\rho)$.
\item If $\Lambda$ is strongly associated with two pairs $(\varphi,\rho)$ and $(\tilde\varphi,\rho)$ which share the density $\rho$, then the full jet of $\tilde\varphi - \varphi$ at $\x$ is a complex constant.
\end{enumerate}
It follows that for any phase-density pair $(\varphi,\rho)$ with a critical point $\x$ there exists a unique FOI $\Lambda$ strongly associated with it and such that $\Lambda_0 = \delta_{\x}$. It is coordinate-independent because it is determined by the coordinate-independent conditions (\ref{E:axiom}) and (\ref{E:strong}). After some preparations, we will give a formula for $\Lambda$ in local coordinates. 

\section{Operators on a space of formal jets}

Let $M$ be a real manifold of dimension $n$. Denote by $\J$ the space of jets of infinite order on $M$ supported at $\x \in M$, which is equipped with the decreasing filtration $\{\J_i\}$ by the order of zero at $\x$. The space~$\J$ is complete with respect to this filtration. Denote by $\D^{(k)}$ the space of differential operators on~$\J$ of order at most $k$. An element $A \in \D^{(k)}$ is a linear mapping $A:\J \to \J$ such that $\ad(f_0) \ldots \ad(f_k)A=0$ for any $f_i \in \J$, where $\ad(f)A = [f,A] = f \circ A - A \circ f$. Then
\[
    \D = \bigcup_{k=0}^\infty \D^{(k)}
\]
is the algebra of differential operators of finite order on~$\J$. The filtration on $\J$ induces a filtration $\{\D_i\}$, where $i \in \Z$, on $\D$. The filtration degree of an operator $A \in \D$ is the largest integer $k$ such that
\[
A\J_r \subset \J_{r+k}
\]
for all $r \geq 0$. The filtration degree of a differential operator of order $k$ is at least~$-k$, $\D^{(k)} \subset \D_{-k}$. Each space $\D^{(k)}$ is complete with respect to this filtration, but $\D$ is not. The completion $\hat\D$ of $\D$ contains differential operators of infinite order on~$\J$.  Denote the filtration degree of $f\in \J$ and of $A \in \D$ by $d(f)$ and $d(A)$, respectively.

Let $\Ncal$ be the algebra of natural operators on $\J[[\nu]]$,
\[
  \Ncal := \{A_0 + \nu A_1 + \ldots | A_r \in\D^{(r)} \mbox{ for all } r \geq 0\}.
\]
Clearly, $\nu^k \Ncal \subset \Ncal$ for all $k \geq 0$. We consider the algebra $\Ncal((\nu))$ whose elements are of the form $\nu^k A$, where $k \in \Z$ and $A \in \Ncal$,
\[
    \Ncal((\nu)) = \bigcup_{r=0}^\infty \nu^{-r}\Ncal.
\]
Notice that $\nu^{-1}\Ncal$ is a Lie algebra with respect to the commutator of operators and $\nu^{-1}\Ncal$ acts on $\Ncal$ by the adjoint action: given $A \in \nu^{-1}\Ncal$ and $B \in \Ncal$, we have $\ad(A)B= [A,B]\in \Ncal$.

We equip the algebra $\Ncal((\nu))$ with the following filtration. We set $d(\nu) =2$. The filtration degree of $A \in \nu^r\Ncal$ written as $A = \nu^r A_0 + \nu^{r+1} A_1 + \ldots$ with $A_k \in \D^{(k)}$ is
\[
d(A)=\inf \{ 2(r+k) + d(A_k) | k \geq 0\}.
\]
Since $d(A_k) \geq -k$, we get that $2(r+k) + d(A_k) \geq 2r+k$. Hence, $d(A) \geq 2r$. We call this filtration on $\Ncal((\nu))$ and a similar filtration on $\J((\nu))$ the {\it standard} filtration. The algebra $\Ncal$ is complete with respect to the standard filtration, $\{\Ncal_i\}$, but $\Ncal((\nu))$ and $\J((\nu))$ are not. Denote by~$\A$ the completion of the algebra $\Ncal((\nu))$ with respect to the standard filtration and by $\F$ the completion of $\J((\nu))$. The algebra $\A$ acts on~$\F$. The elements of $\A$ and~$\F$ can be written as certain series
\[
    \sum_{r \in \Z} \nu^r A_r \mbox{ and } \sum_{r \in \Z} \nu^r f_r,
\]
respectively, where $A_r \in \hat\D$ and $f_r \in \J$. Set
\[
    \g := \{A \in \nu^{-1}\Ncal | d(A) \geq 1\} \subset \A_1.
\]
It is a pronilpotent Lie algebra whose Lie group $\exp \g$ lies in $\A_0$.

Suppose that $(\varphi,\rho)$ is a phase-density pair on $M$ with a critical point $\x$ and $U$ is a coordinate neighborhood of $\x$ with coordinates $\{x^i\}$ such that $x^i(\x)=0$ for all $i$, that is, $\x=0$. We set
\[
         h_{ij} := \frac{\p \varphi_{-1}}{\p x^i \p x^j}\bigg|_{x=0}.
\]
Then $(h_{ij})$ is a symmetric nondegenerate complex matrix with constant entries. Let $(h^{ij})$ be its inverse matrix. We set
\begin{equation}\label{E:lapl}
    \psi := \frac{1}{2} h_{ij} x^i x^j  \mbox{ and }  \Delta := -\frac{1}{2} h^{ij} \frac{\p^2}{\p x^i \p x^j}.
\end{equation}
In \cite{LMP10}, Lemma 9.1, we proved that the formal distribution
\begin{equation}\label{E:tillam}
   \tilde\Lambda(f) := e^{\nu\Delta}f \big |_{x=0}
\end{equation}
is a FOI associated with the pair $(\nu^{-1}\psi, dx)$, where $dx=dx^1 \ldots dx^n$ is the Lebesgue density on $U$.
\begin{lemma}\label{L:tillam}
The FOI (\ref{E:tillam}) is strongly associated with the pair $(\nu^{-1}\psi, dx)$.
\end{lemma}
\begin{proof}
It follows from formula (\ref{E:axiom}) with $v = x^i \p_i$ and $\rho=dx$ that
\begin{equation}\label{E:calone}
    \tilde\Lambda\left(x^i \p_i f + \left(2\nu^{-1}\psi + n \right)f\right)=0,
\end{equation}
where we have used that $v\psi=2\psi$ and $\mathbb{L}_v \rho=n\rho$. Replacing $f$ with  $- \frac{1}{2} h^{ij} \p_j f$ and setting $v = \p_i$ in (\ref{E:axiom}), we get
\begin{equation}\label{E:caltwo}
    \tilde\Lambda\left(\Delta f - \frac{1}{2}\nu^{-1} x^i\p_i f\right)=0,
\end{equation}
where the summation on $i$ is assumed. Dividing (\ref{E:calone}) by $2\nu$ and adding the result to (\ref{E:caltwo}), we get
\begin{equation}\label{E:calthree}
    \tilde\Lambda\left(\Delta f + \left(\nu^{-2}\psi + \frac{1}{2}\nu^{-1} n \right) f \right)=0.
\end{equation}
Now we verify (\ref{E:strong}) with $\varphi= \nu^{-1}\psi$ and $\rho=dx$ using (\ref{E:calthree}):
\begin{eqnarray*}
\frac{d}{d\nu}\tilde\Lambda (f)- \tilde\Lambda\left(\frac{df}{d\nu} - \left(\nu^{-2} \psi + \frac{1}{2}\nu^{-1} n\right)f\right) = \\
\tilde\Lambda\left( \Delta f + \frac{\p f}{\p \nu}\right) - \tilde\Lambda\left(\frac{df}{d\nu} - \left(\nu^{-2} \psi + \frac{1}{2}\nu^{-1} n\right)f\right) =0.
\end{eqnarray*}
\end{proof}
Assume that locally
\[
     \rho = e^u \, dx,
\]
where $u = u_0 + \nu u_1 + \ldots \in C^\infty(U)[[\nu]]$.  We call the function 
\[
    \chi(x) := \varphi (x) - \nu^{-1}\psi - \varphi_0(0) + u(x) - u_0(0)
\]
the {\it phase remainder}. Since we will need only the jet of infinite order of $\chi$ at $\x=0$, we identify $\chi=\nu^{-1}\chi_{-1}+\chi_0+ \ldots$ with its jet. The order of zero of $\chi_{-1}$ and of $\chi_0$ at $\x=0$ is at least 3 and 1, respectively. Hence, $\chi \in \F_1$ and therefore the operator $\exp \chi$ acts on $\F_0$. Since $d(\nu\Delta)=0$, the operator $\exp (\nu\Delta)$ acts on $\J((\nu))$ and respects the standard filtration. Thus, it also acts on $\F$ respecting the filtration. We define a formal distribution $\Lambda$ on $U$ supported at $\x=0$ by the formula
\begin{equation}\label{E:distrlam}
   \Lambda(f) := \left(e^{\nu\Delta} e^{\chi} f\right)\big |_{x=0}.
\end{equation}
If $f \in C^\infty(U)[[\nu]]$, then its jet at $\x=0$ lies in $\F_0$. Hence, $e^{\nu\Delta} e^{\chi} f \in \F_0$, which implies that $\Lambda(f)\in \C[[\nu]]$ and therefore $\Lambda = \Lambda_0 + \nu \Lambda_1 +\ldots$ (the coefficients at the negative powers of $\nu$ in $e^{\nu\Delta} e^{\chi} f$ vanish at $\x=0$ because its filtration degree is nonnegative).

\begin{proposition}\label{P:lamfoi}
The formal distribution (\ref{E:distrlam}) is the unique FOI $\Lambda = \Lambda_0 + \nu \Lambda_1 +\ldots$ strongly associated with the pair $(\varphi,\rho)$ and such that $\Lambda_0=\delta$.
\end{proposition}
\begin{proof}
It follows from \cite{LMP10}, Theorem 9.1, that $\Lambda$ is associated with the pair $(\varphi,\rho)$ and $\Lambda_0=\delta$. It remains to prove that it is strongly associated with $(\varphi,\rho)$ or, equivalently, with the pair $(\nu^{-1}\psi + \chi, dx)$. We will use Lemma \ref{L:tillam} and the fact that $\Lambda(f) = \tilde\Lambda(e^{\chi} f)$. We have
\begin{eqnarray*}
 \frac{d}{d\nu}\Lambda(f) =\frac{d}{d\nu}\tilde\Lambda(e^\chi f) = \tilde\Lambda\left(\frac{d}{d\nu}(e^\chi f)  + \left(-\frac{\psi}{\nu^2}  - \frac{n}{2\nu}\right)(e^\chi f)  \right)=\\
 \tilde\Lambda \left(e^\chi\left(\frac{df}{d\nu} + \left(-\frac{\psi}{\nu^2} + \frac{d\chi}{d\nu} - \frac{n}{2\nu}\right)f \right)\right) = \\
 \Lambda \left(\frac{df}{d\nu} + \left(\frac{d}{d\nu}(\nu^{-1}\psi + \chi)  - \frac{n}{2\nu}\right)f \right).
\end{eqnarray*}
\end{proof}

\section{Identification of formal oscillatory integrals}

Below we will prove the following theorem.
\begin{theorem}\label{T:foiosc}
  A formal distribution $\Lambda = \Lambda_0 + \nu \Lambda_1 +\ldots$ on a manifold $M$ supported at a point $\x\in M$ is a FOI strongly associated with some pair $(\varphi,\rho)$  with the critical point $\x$ and such that $\Lambda_0=\delta_{\x}$ if and only if $\Lambda$ is a nondegenerate oscillatory distribution.
\end{theorem}

Let $(h_{ij})$ be a symmetric nondegenerate complex $n \times n$ matrix with constant entries and $(h^{ij})$ be its inverse matrix. We use the same notations $\psi$ and $\Delta$ as in (\ref{E:lapl}). Observe that $\nu\Delta$ and $\nu^{-1}\psi$ lie in $\nu^{-1}\Ncal$ and $d(\nu\Delta)=d(\nu^{-1}\Ncal)=0$.
\begin{lemma}\label{L:conje}
The adjoint action of the operators $\nu\Delta$ and $\nu^{-1}\psi$ by derivations of the algebra $\Ncal$ integrates to automorphisms of this algebra which respect the standard filtration and therefore extend to automorphisms of the algebras $\A$ and $\g$ and the Lie group $\exp{\g}$.
\end{lemma}
\noindent{\it Remark.} The operator $\exp{\nu\Delta}$ acts on the space $\F$, but the operator $\exp(\nu^{-1}\psi)$ is undefined on that space.

\begin{proof}
Given $A = A_0 + \nu A_1 +\ldots \in \Ncal$, we have $d(A_r) \geq -r$, hence $d(\nu^rA_r) \geq r$, and therefore $\nu^rA_r \in \Ncal_r$ for all $r \geq 0$. 
The action of $\exp (\ad(\nu\Delta))$ maps $\nu^r A_r$ to
\[
  e^{\ad (\nu\Delta)}(\nu^rA_r)= \sum_{s=0}^\infty \frac{\nu^{r+s}}{s!} (\ad(\Delta))^s (A_r)\in \Ncal_r.
\]
The action of $\exp(\ad (\nu^{-1}\psi))$ maps $\nu^r A_r$ to
\[
   e^{\ad (\nu^{-1}\psi)}(\nu^rA_r)= \sum_{s=0}^r \frac{1}{s!} (\ad(\nu^{-1}\psi))^s (\nu^rA_r)\in \Ncal_r.
\]
It follows that $e^{\ad(\nu\Delta)}(A)$ and $e^{\ad (\nu^{-1}\psi)}(A)$ are elements of $\Ncal$, because $\Ncal$ is complete with respect to the standard filtration. 
\end{proof}

Now we will give a proof of Theorem \ref{T:foiosc}.
\begin{proof}
Fix local coordinates $\{x^i\}$ around $\x$ such that $x^i(\x)=0$ for all~$i$. Denote by $\mathfrak{b}$ the Lie algebra of operators $A \in \g$ such that $\delta \circ A=0$ and by $\mathfrak{c}$ the Lie algebra of operators from $\g$ with constant coefficients. Then $\g = \mathfrak{b}\oplus\mathfrak{c}$. Let $(\varphi,\rho)$ be a phase-density pair on $M$ with the critical point $\x=0$ and $\chi$ be the corresponding phase remainder. Then (\ref{E:distrlam}) is the unique FOI strongly associated with $(\varphi,\rho)$ and such that $\Lambda_0=\delta$. Lemma \ref{L:conje} implies that
\[
     e^{\ad(\nu\Delta)} (e^{\chi}) \in \exp\g.
\]
By Proposition \ref{P:factor}, there exist unique elements $B \in \mathfrak{b}$ and $C \in \mathfrak{c}$ such that
\begin{equation}\label{E:ebc}
  e^{\ad(\nu\Delta)} (e^{\chi}) = e^B e^C.
\end{equation}
It follows that
\begin{eqnarray*}
   \Lambda(f) = \left(e^{\nu\Delta} e^{\chi} f\right)\big |_{x=0}= \left(e^{\nu\Delta} e^{\chi} e^{-\nu\Delta} e^{\nu\Delta} f\right)\big |_{x=0}=\hskip 1.7cm\\
    \left(e^{\ad(\nu\Delta)} (e^{\chi}) e^{\nu\Delta} f\right)\big |_{x=0}=\left(e^B e^C e^{\nu\Delta} f\right)\big |_{x=0} = \left(e^{\nu\Delta+C} f \right) |_{x=0},
\end{eqnarray*}
where we have used that the operators with constant coefficients $\nu\Delta$ and $C$ commute. The operator $C$ can be written as
\[
   C = \nu^{-1}(X_0 + \nu X_1 + \ldots),
\]
where $X_r$ has constant coefficients, is of order at most $r$, and whose filtration degree is at least $3-2r$ for all $r$. It follows that $X_0=X_1=0$ and $X_2$ is of order at most 1. We see that
\[
   \nu\Delta+C = \nu^{-1} \left(\nu^2(\Delta+ X_2) + \nu^3 X_3 + \nu^4 X_4 + \ldots\right) \in \nu^{-1}\Ncal
\]
and the operator $\Delta+ X_2$ can be written in coordinates as
\begin{equation}\label{E:xtwo}
    - \frac{1}{2} h^{ij} \p_i \p_j + b^i\p_i +c.
\end{equation}
Since the matrix $(h^{ij})$ is nondegenerate, the FOI $\Lambda$ is a nondegenerate oscillatory distribution.

Now suppose that $\Lambda$ is a nondegenerate oscillatory distribution on a manifold $M$ supported at $\x\in M$. Fix local coordinates $\{x^i\}$ around $\x$ such that $x^i(\x)=0$ for all~$i$. According to Proposition \ref{P:osciff}, there exists a unique natural operator with constant coefficients $X = \nu^2 X_2 + \nu^3 X_3 + \ldots$ such that
\[
   \Lambda = \delta \circ \exp (\nu^{-1}X).
\]
If we write $X_2$ as (\ref{E:xtwo}), where $(h^{ij})$ is a symmetric matrix with constant entries, then this matrix is nondegenerate because $\Lambda$ is a nondegenerate oscillatory distribution. We will have that
\[
   C:= \nu^{-1}X + \frac{\nu}{2} h^{ij}\p_i \p_j \in \mathfrak{c}.
\]
Let $(h_{ij})$ be the matrix inverse to $(h^{ij})$. We will use the settings (\ref{E:lapl}) and will show that there exists a $\nu$-formal jet $\chi = \nu^{-1}\chi_{-1} + \chi_0 + \ldots$ at $\x=0$ of positive filtration degree such that (\ref{E:ebc}) holds for some $B \in \mathfrak{b}$. It will mean that $\Lambda$ is a FOI at $\x=0$ strongly associated with the phase-density pair $(\nu^{-1}\psi + \chi, dx)$\footnote{By Borel's lemma it suffices to give only the jet of infinite order of the phase at $\x=0$.}.

Denote by $\mathfrak{e}$ the Lie algebra of operators from $\g$ that can be written as
\[
   A= \p_i \circ A^i
\]
for some formal differential operators $A_i$. If we use the standard transposition $A \mapsto A^t$ of differential operators such that $(\p_i)^t = - \p^t$ and $(x^i)^t = x^i$, then $A \in\mathfrak{e}$ if $A \in \g$ and $A^t$ annihilates constants, $A^t 1=0$. Denote by $\mathfrak{f}$ the Lie algebra of multiplication operators from $\g$. Then $\g = \mathfrak{e}\oplus\mathfrak{f}$. A simple calculation shows that
\[
   e^{-\ad(\nu\Delta)} (x^k) = x^k + \nu h^{kl} \frac{\p}{\p x^l} \mbox{ and } e^{\ad(\nu^{-1}\psi)}e^{-\ad(\nu\Delta)} (x^k) = \nu h^{kl} \frac{\p}{\p x^l}.
\]
Therefore, the conjugation
\[
   A \mapsto e^{\ad(\nu^{-1}\psi)}e^{-\ad(\nu\Delta)} (A)
\]
provides isomorphisms of the Lie algebra $\mathfrak{b}$ onto $\mathfrak{e}$ and of the Lie group $\exp\mathfrak{b}$ onto $\exp\mathfrak{e}$. By Proposition \ref{P:factor}, there exist unique elements $E \in \mathfrak{e}$ and $\chi \in \mathfrak{f}$ such that
\[
     e^{\ad(\nu^{-1}\psi)}\left(e^C\right ) = e^E e^\chi.
\]
Acting on both sides by $\exp(\ad(\nu\Delta))\exp(-\ad(\nu^{-1}\psi))$, we get
\[
    e^C= \left(e^{\ad(\nu\Delta)} e^{\ad(-\nu^{-1}\psi)}\left(e^E\right ) \right) \left(e^{\ad(\nu\Delta)} (e^\chi) \right),
\]
which implies (\ref{E:ebc}) if we set
\[
    B:= - e^{\ad(\nu\Delta)} e^{\ad(-\nu^{-1}\psi)}\left(E\right )   \in \mathfrak{b}.
\]
It completes the proof of the theorem.
\end{proof}
It is interesting to notice that Theorem \ref{T:foiosc} and Proposition 3.1 in \cite{LMP10} imply that if $\Lambda$ is a nondegenerate oscillatory distribution supported at $\x$, then the pairing
\[
    f,g \mapsto \Lambda(fg)
\]
on the space of formal jets $\J[[\nu]]$ is nondegenerate.

One of the consequences of Theorems \ref{T:grosc} and \ref{T:foiosc} is that Fedosov's star product is given by some formal oscillatory integral (the distribution $\Lambda_x$ for a Fedosov's star product is nondegenerate for any $x$ because $C_1(f,g) = \pi^{ij}\p_i f \p_j g$, where $\pi^{ij}$ is a nondegenerate Poisson tensor). However, Fedosov's construction does not use any oscillatory integral formulas. Only in the simplest case of the Moyal-Weyl star product it is given by the asymptotic expansion of a known oscillatory integral (and hence by a formal oscillatory integral).

\end{document}